\documentclass[a4paper,11pt]{amsart}
\usepackage{amssymb,amscd,amsxtra,xypic}
\usepackage[all]{xy}
\usepackage{array}
\usepackage{nicematrix}
\usepackage{textcomp}
\usepackage{bm}
\usepackage{mathtools}

\usepackage[pagebackref=true, linktocpage=true]{hyperref}
\hypersetup{%
bookmarksnumbered=true,%
colorlinks=true,%
linkcolor=blue,%
citecolor=blue,%
urlcolor=blue,%
setpagesize=false,%
pdftitle={},%
pdfauthor={Takuzo Okada}}

\theoremstyle{plain}
\newtheorem{Thm}{Theorem}[section]
\newtheorem{Lem}[Thm]{Lemma}

\newtheorem{Prop}[Thm]{Proposition}

\theoremstyle{definition}
\newtheorem{Def}[Thm]{Definition}
\newtheorem{Cond}[Thm]{Condition}
\newtheorem{Rem}[Thm]{Remark}
\newtheorem*{Ack}{Acknowledgments}

\numberwithin{equation}{section}

\newcommand{\Aut}{\operatorname{Aut}}
\newcommand{\Proj}{\operatorname{Proj}}

\newcommand{\Spec}{\operatorname{Spec}}
\newcommand{\Cl}{\operatorname{Cl}}

\newcommand{\bMov}{\overline{\operatorname{Mov}}}
\newcommand{\Nef}{\operatorname{Nef}}

\newcommand{\wt}{\operatorname{wt}}

\newcommand{\ord}{\operatorname{ord}}

\newcommand{\Qsm}{\operatorname{Qsm}}

\renewcommand{\wt}{\operatorname{wt}}

\newcommand{\mbA}{\mathbb{A}}

\newcommand{\mbC}{\mathbb{C}}

\newcommand{\mbP}{\mathbb{P}}
\newcommand{\mbQ}{\mathbb{Q}}
\newcommand{\mbR}{\mathbb{R}}
\newcommand{\mbT}{\mathbb{T}}

\newcommand{\mbZ}{\mathbb{Z}}

\newcommand{\mcF}{\mathcal{F}}

\newcommand{\mcM}{\mathcal{M}}

\newcommand{\mcO}{\mathcal{O}}

\newcommand{\msF}{\mathsf{F}}
\newcommand{\msG}{\mathsf{G}}
\newcommand{\msH}{\mathsf{H}}

\newcommand{\msp}{\mathsf{p}}
\newcommand{\msq}{\mathsf{q}}
\newcommand{\msr}{\mathsf{r}}

\newcommand{\ratmap}{\dashrightarrow}

\makeatletter
\def\imod#1{\allowbreak\mkern10mu({\operator@font mod}\,\,#1)}
\makeatother

\title[WCIs of type $(12, 14)$ in $\mbP (1, 2, 3, 4, 7, 11)$]{Birational geometry of weighted complete intersections of type $(12, 14)$ in $\mbP (1, 2, 3, 4, 7, 11)$}
\author[Takuzo Okada]{Takuzo Okada}
\address{Faculty of Mathematics, Kyushu University, Fukuoka 819-0395, Japan}
\email{tokada@math.kyushu-u.ac.jp}
\subjclass[2020]{14J45 \and 14E08}
\keywords{Fano variety; Birational solidity; Rationality questions}
\date{}

\begin{document}

\begin{abstract}
We show that any quasismooth Fano threefold weighted complete intersections of type $(12, 14)$ in $\mathbb{P} (1, 2, 3, 4, 7, 11)$ is birationally solid.
\end{abstract}

\maketitle

\tableofcontents

\section{Introduction} \label{sec:intro}

Throughout the paper, the ground field is assumed to be the complex number field $\mbC$.
Rationality problem of Fano $3$-folds has been studied extensively.
We refer readers to \cite{Pro19}, \cite{Pro22}, \cite{Pro23}, \cite{Pro24} and \cite{Pro25} for recent systematic studies by Prokhorov.
In this paper we focus on Fano $3$-folds embedded as quasismooth and well-formed complete intersections in weighted projective spaces.
We simply call such a variety as a quasismooth and well-formed Fano $3$-fold WCI, for short.
For a quasismooth and well-formed Fano $3$-fold WCI $X$, the class group of $X$ is isomorphic to $\mbZ$ (\cite[Theorem~3.2.4]{Dol}) and it is generated by the Weil divisor class $A$ corresponding to $\mcO_X (1)$.
The number $\iota_X$ such that $-K_X = \iota_X A$ in $\Cl (X)$ is called the \textit{index} of $X$.

After the works of Iskovskikh--Manin \cite{IM}, Iskovskikh \cite{Isk} and Corti--Pukhlikov--Reid \cite{CPR}, it is finally settled by Cheltsov--Park \cite{CP} that any quasismooth Fano 3-fold weighted hypersurface of index $1$ is birationally rigid, which implies that it is irrational.
A Fano variety $X$ of Picard number $1$ is said to be \textit{birationally rigid} if $X$ is not birational to a Mori fiber space other than $X$.
A generalized notion of birational solidity was introduced by Abban--Okada \cite{AO} and by Shokurov independently: a Fano variety $X$ of Picard number $1$ is \textit{birationally solid} if $X$ is not birational to a Mori fiber space over a positive dimensional base.
Note that birational solidity also implies irrationality.
The classification of birationally solid quasismooth Fano threefold weighted hypersurfaces is completed by \cite{OkSolid}.

The next stage is to consider the case of codimension $2$.
There are $85$ families of quasismooth and well-formed Fano $3$-fold WCIs of codimension $2$ and index $1$.
They are studied in \cite{IP96}, \cite{OkI} and \cite{AZ}, and it is in particular proved that $19$ families consist of birationally rigid members and $6$ families consist of birationally non-solid members.
The members in the remaining $60$ families are conjectured to be birationally solid and this is confirmed for many families by \cite{CM04}, \cite{OkII} and \cite{OkIII}.
Duarte Guerreiro \cite{DG23} investigated quasismooth and well-formed Fano $3$-fold WCIs of codimension $2$ and index greater than $1$ that form $40$ families, and it is in particular proved that none of them are birationally rigid.
To be more precise, it is proved that $18$ families consist of birationally non-solid members and the members of the remaining $22$ families are conjectured to be birationally solid (\cite[Conjecture~1.8]{DG23}).
The aim of this paper is to consider a specific family of Fano $3$-fold WCIs of codimension $2$ and index $2$ which is one of the above mentioned $22$ families, and show that the conjecture holds true for this family. 
This provides a first example of birationally solid Fano $3$-fold WCI of codimension $\ge 2$ and index $\ge 2$.

\begin{Thm} \label{mainthm}
Let $X = X_{12, 14} \subset \mbP (1, 2, 3, 4, 7, 11)$ be a quasismooth Fano $3$-fold weighted complete intersection of type $(12, 14)$.
Then $X$ is birational to a Fano $3$-fold weighted hypersurface $\hat{X} \subset \mbP (1, 1, 1, 2, 3)$ of degree $7$ admitting a terminal singularity of type $cE_6$, and $X$ is not birational to any other Mori fiber space.
In particular, $X$ is birationally solid and  $X$ is not rational.
\end{Thm}

We explain the outline of the paper.
Let $X$ be a Fano $3$-fold in Theorem~\ref{mainthm}.
An elementary link $\sigma \colon X \ratmap \hat{X}$ to a (non-quasismooth) Fano $3$-fold weighted hypersurface $\hat{X}$ of degree $7$ in $\mbP (1, 1, 1, 2, 3)$  is constructed in \cite{DG23}.
We recall the construction of $\sigma$ and study $\hat{X}$ in detail in \S \ref{sec:pfmainthm}.
In view of the Sarkisov program \cite{Corti95}, the proof of Theorem~\ref{mainthm} will be done by classifying the elementary links from $X$ and from $\hat{X}$.
Half of these tasks are already done in \cite{DG23} where it is proved that $\sigma$ is the only elementary link from $X$.
We complete the classification of the elementary links from $\hat{X}$ in \S \ref{sec:hypdeg7} and \S \ref{sec:pfmainthm}.
To complete the classification, we recall the notion of maximal extraction and explain preliminary results in \S \ref{sec:prelim}.

\begin{Ack}
The author was supported by JSPS KAKENHI Grant Numbers JP22H01118 and JP24K00519.
\end{Ack}

\section{Preliminaries}
\label{sec:prelim}

In this paper, by a divisorial contraction $\varphi \colon Y \to X$, we mean a contraction of a $K_Y$-negative extremal ray from a normal variety $Y$ with only terminal singularities that contracts a prime divisor.

Let $f = \sum \alpha_I x^I \in \mbC [x_1, \dots, x_n]$ be a polynomial.
For a monomial $x^J = x_1^{j_1} \cdots x_n^{j_n}$, we write $x^J \in f$ if the coefficient of $x^J$ in $f$ is nonzero.
Let $\bm{w}$ be a weight on the variables $x_1, \dots, x_n$ defined as $\bm{w} (x_1, \dots, x_n) = (a_1, \dots, a_n)$ for some nonnegative rational numbers $a_i$.
Then, for an integer $d$, we define
\[
f_{\bm{w} = d} := \sum_{\bm{w} (x^I) = d} \alpha_I x^I.
\]

\subsection{Maximal singularities}

\begin{Def}
Let $X$ be a Fano variety of Picard number $1$.
Let $\varphi \colon Y \to X$ be a divisorial contraction and let $E$ be its exceptional divisor.
We say that $\varphi$ is a \textit{maximal extraction} if there is a movable linear system $\mcM$ such that
\[
\ord_E (\mcM) > n a_X (E),
\]
where 
\begin{itemize}
\item $n$ is the positive rational number such that $\mcM \sim_{\mbQ} - n K_X$, 
\item $\ord_E (\mcM) = \ord_E (\varphi^*\mcM) := \min \{\, \ord_E \varphi^*D \mid D \in \mcM \, \}$, and
\item $a_X (E) := \ord_E (K_Y-\varphi^*K_X)$ is the discrepancy of $X$ at $E$.
\end{itemize}
We say that $\varphi$ is a \textit{Sarkisov extraction} if it initiates an elementary link.
A subvariety $\Gamma \subset X$ is a \textit{maximal center} (resp.\ \textit{Sarkisov center}) if there is a maximal extraction (resp.\ Sarkisov extraction) whose center is $\Gamma$.
\end{Def}

We have the following implications for these notions of extractions.

\begin{Lem}[{\cite[Lemma~2.5]{Ok3}}]
Let $\varphi \colon Y \to X$ be a divisorial contraction to a Fano variety $X$ of Picard number $1$.
If $\varphi$ is a Sarkisov extraction, then it is a maximal extraction.
\end{Lem}

\subsection{Weighted projective spaces and rank $2$ toric varieties}

Let
\[
\mbP := \mbP (a_0, \dots, a_n) = \Proj \mbC [x_0, \dots, x_n]
\]
be a weighted projective space with homogeneous coordinates $x_0, \dots, x_n$ of weights $a_0, \dots, a_n$, respectively.
As an ambient space of Fano $3$-fold WCIs, we always assume that the weighted projective space $\mbP$ is \textit{well-formed} that is, the greatest common divisor of any $n$ of $a_0, \dots, a_n$ is $1$.
For a coordinate $\xi \in \{x_0, \dots, x_n\}$, we denote by $\msp_{\xi} = (0\!:\!\cdots\!:\!0\!:\!1\!:\!0\!:\!\cdots\!:\!0) \in \mbP$ the point at which only the coordinate $\xi$ does not vanish. 
For a quasi-homogeneous polynomials $f_1, \dots, f_m \in \mbC [x_0, \dots, x_n]$, we define
\[
(f_1 = \cdots = f_m = 0) := \Proj \mbC [x_0, \dots, x_n]/(f_1, \dots, f_m).
\]
We sometimes put $\mbP$ as a subscript and denote $(f_1 = \cdots = f_m)_{\mbP}$ when we make explicit the ambient weighted projective space $\mbP$.

\begin{Def}
Let $V$ be a closed subscheme of $\mbP$ defined by a homogeneous ideal $I \subset \mbC [x_0, \dots, x_n]$.
The \textit{quasismooth locus} $\Qsm (V)$ of $V$ is defined to be the image of the smooth locus of $C_V^* := C_V \setminus \{o\}$ under the natural morphism $\mbA^{n+1} \setminus \{o\} \to \mbP$, where $C_V = \Spec \mbC [x_0, \dots, x_n]/I$ is the affine cone of $V$ and $o \in \mbA^{n+1}$ is the origin.
For a subset $S \subset V$, we say that $V$ is \textit{quasismooth along} $S$ if $S \subset \Qsm (V)$.
We say that $V$ is \textit{quasismooth} if $V = \Qsm (V)$.
\end{Def}

We sometimes denote $\mbP (a_x, b_y, c_z, d_t, \dots)$, where $a, b, c, d, \dots$ are positive integers, and this means that it is the weighted projective space with homogeneous coordinates $x, y, z, t, \dots$ of weights $a, b, c, d, \dots$, respectively.

We next recall the definition of rank $2$ toric varieties which will be useful when we construct various links.
Let $3 \leq m + 1 < n$ be positive integers and $a_1, \dots, a_n, b_1, \dots, b_n$ be integers.
Let $R = \mbC [x_1, \dots, x_n]$ be the polynomial ring with a $\mbZ^2$-grading defined by the $2 \times n$ matrix
\[
\begin{pmatrix}
a_1 & a_2  & \cdots & a_n \\
b_1 & b_2 & \cdots & b_n
\end{pmatrix}, 
\]
that is, the bi-degree of the variable $x_i$ is $(a_i, b_i) \in \mbZ^2$.
We denote by
\[
\mbT := \mbT \begin{pNiceArray}{ccc|ccc}[first-row]
x_1 & \cdots & x_m & x_{m+1} & \cdots & x_n \\
a_1 & \cdots & a_m & a_{m+1} & \cdots & a_n \\
b_1 & \cdots & b_m & b_{m+1} & \cdots & b_n
\end{pNiceArray}
\]
the toric variety whose Cox ring is $R$ and the irrelevant ideal is $I = (x_1, \dots, x_m) \cap (x_{m+1}, \dots, x_n)$.
It is the geometric quotient
\[
\mbT = (\mbA^n_{x_1, \dots, x_n} \setminus V (I))/(\mbC^*)^2,
\]
where the $(\mbC^*)^2$-action is given by
\[
(\lambda, \mu) \cdot (x_1, \dots, x_n) = (\lambda^{a_1} \mu^{b_1} x_1, \dots, \lambda^{a_n} \mu^{b_n} x_n).
\]
The variety $\mbT$ is a simplicial toric variety of Picard number $2$.
Let $\msp \in \mbT$ be a point and let $\msq = (\alpha_1, \dots, \alpha_n) \in \mbA^n$ be a preimage of $\msp$ by the morphism $\mbA^n \setminus V (I) \to \mbT$.
In this case we express $\msp$ as
\[
\msp = (\alpha_1\!:\!\cdots\!:\!\alpha_m \mid \alpha_{m+1}\!:\!\cdots\!:\!\alpha_n) \in \mbT.
\]
For polynomials $f_1, \dots, f_m$ which are quasi-homogeneous with respect to the above $(\mbC^*)^2$-action, we denote by
\[
(f_1 = \cdots = f_m = 0) \subset \mbT
\]
the closed subscheme which is the image of $(f_1 = \cdots = f_m = 0) \subset \mbA^n \setminus V (I)$.
We sometimes denote $(f_1 = \cdots = f_m = 0)_{\mbT}$ when we make explicit the ambient toric variety $\mbT$.

\begin{Rem}[{\cite[Remark~2.13]{OkSolid}}]
Let $\mbP = \mbP (a_0, \dots, a_n)$ be the weighted projective space with homogeneous coordinates $x_0, \dots, x_n$ of weights $a_0, \dots, a_n$, respectively.
Let $b_1, \dots, b_n$ be positive integers.
Then the morphism
\[
\Psi \colon \mbT := \mbT \begin{pNiceArray}{cc|cccc}[first-row]
u & x_0 & x_1 & x_2 & \dots & x_n \\
0 & a_0 & a_1 & a_2 & \dots & a_n \\
-a_0 & 0 & b_1 & b_2 & \dots & b_n
\end{pNiceArray} \to
\mbP,
\]
defined by
\[
(u\!:\!x_0 \, | \, x_1\!:\!\cdots\!:\!x_n) \mapsto (x_0\!:\!u^{b_1/a_0} x_1\!:\!u^{b_2/a_0} x_2\!:\!\cdots\!:\!u^{b_n/a_0} x_n)
\]
is the weighted blow-up of $\mbP$ at the point $\msp_{x_0}$ with $\wt (x_1, \dots, x_n) = \frac{1}{a_0} (b_1, \dots, b_n)$.
The $\Psi$-exceptional divisor is the divisor $(u = 0)_{\mbT}$ on $\mbT$.
\end{Rem}

\subsection{Special divisors over a $3$-fold germ}

Let $\msp \in X$ be the germ of an algebraic variety.
A \textit{divisor $E$ over $X$} is a prime divisor $E$ on a normal variety $Y$ admitting a birational morphism $Y \to X$.
We say that divisors $E_1$ and $E_2$ over $X$ are \textit{equivalent}, denoted by $E_1 \sim_{\mathrm{val}} E_2$, if their associated valuations coincide.
A \textit{divisor $E$ of discrepancy $a$ over a point $\msp \in X$} is a divisor $E$ over $X$ such that its center on $X$ is $\msp$ and $a_X (E) = a$. 
When we count the number of divisors over a germ $\msp \in X$, we always count them up to the equivalence relation $\sim_{\mathrm{val}}$. 

Let $a_1, \dots, a_n \ge 0$ and $r \geq 2$ be integers.
We consider the action of $\mbZ_r := \mbZ/r \mbZ$ on $\mbA^n_{x_1, \dots, x_n}$ by 
\[
(x_1, \dots, x_n) \mapsto (\zeta^{a_1} x_1, \dots, \zeta^{a_n} x_n),
\]
where $\zeta \in \mbC$ is a primitive $r$th root of unity.
For $\mbZ_r$-semi-invariant polynomials $f_1, \dots, f_m \in \mbC [x_1, \dots, x_n]$, the quotient by the $\mbZ_r$-action of the closed subscheme $\Spec \mbC [x_1, \dots, x_n]/(f_1, \dots, f_m)$ of $\mbA^n$ is denoted by
\[
(f_1 = \cdots = f_m = 0)/\mbZ_r ({a_1}_{x_1}, \dots, {a_n}_{x_n}).
\]

\begin{Lem} \label{lem:cA/2divsmin}
Let $\msp \in X$ be the germ of a $3$-fold terminal singularity of type $cA/2$ which is analytically equivalent to 
\[
\bar{o} \in (x y + g (z^2, t) = 0)/\mbZ_2 (1_x, 1_y, 1_z, 0_t),
\]
where  the weighted order of $g (z^2, t)$ with respect to $\wt (z, t) = (1, 2)$ is $6$ and $t^3 \in g (z^2, t)$.
Then, there are $3$ divisors of discrepancy $1/2$ over $\msp \in X$.
\end{Lem}

\begin{proof}
Let $\varphi \colon Y \to X$ be the weighted blowup of $\msp \in X$ with weights $\wt (x, y, z, t) = \frac{1}{2} (5, 1, 1, 2)$.
The exceptional locus of $\varphi$ is isomorphic to the hypersurface
\[
(x y + g_{\wt = 6} (z^2, t) = 0) \subset \mbP (5_x, 1_y, 1_z, 2_t),
\]
where $g_{\wt = 6} (z^2, t)$ is the weighted order $6$ terms in $g (z^2, t)$ with respect to $\wt (z, t) = (1, 2)$.
The polynomial $x y + g_{\wt = 6} (z^2, t)$ is irreducible since $t^3 \in g_{\wt = 6}$.
We denote this exceptional divisor by $E$.
Then, $E$ is a divisor of discrepancy $1/2$ over $\msp \in X$.

We set $\msq := (1\!:\!0\!:\!0\!:\!0) \in E$, which is the unique point of $Y$ at which $E$ is not Cartier since $t^3 \in g_{\wt = 6} (z^2, t)$.
Let $F$ be a divisor of discrepancy $1/2$ over $\msp \in X$ other than $E$.
Then, its center $C_Y (F)$ on $Y$ is a proper closed subvariety of $E$ and we have
\[
\frac{1}{2} = a_X (F) = a_Y (F) + \frac{1}{2} \ord_F (E).
\]
Hence $C_Y (F)$ must be a non-Gorenstein singular point on $E$ at which $E$ is not Cartier, that is, $C_Y (F)$ is the point $\msq$.
Moreover, we must have $a_Y (F) < 1$.

We work on the $x$-chart $U_x \subset Y$ of the weighted blowup:
\[
U_x = (\tilde{y} + \tilde{g} (\tilde{x}, \tilde{z}, \tilde{t}^2)/\tilde{x}^3 = 0)/\mbZ_5  (3_{\tilde{x}}, 1_{\tilde{y}}, 2_{\tilde{z}}, 1_{\tilde{t}}),
\]
where
\[
\tilde{g} (\tilde{x}, \tilde{z}, \tilde{t}^2) := g (\tilde{z} \tilde{x}, \tilde{t}^2 \tilde{x})/\tilde{x}^3.
\]
The exceptional divisor $E \cap U_x$ is defined by $\tilde{x}$ on this chart.
The point $\msq$ corresponds to the origin of $U_x$ and the singularity $\msq \in Y$ is of type $\frac{1}{5} (1, 2, 3)$.
We can choose $\tilde{x}, \tilde{z}, \tilde{t}$ as local orbifold coordinates of $U_x$ at $\msq$.
For $i = 1, 2, 3, 4$, let $\psi_i \colon W_i \to Y$ be the weighted blowup of $\msq \in Y$ with weights $\wt (\tilde{x}, \tilde{z}, \tilde{t}) = ([3i], [2i], i)$, where $0 \le [m] < 5$ denotes the integer which is congruent to $m$ modulo $5$, and let $F_i$ be the exceptional divisor of $\psi_i$.
We have $a_Y (F_i) = i/5$ for $1 \le i \le 4$ and the divisors $F_1, \dots, F_4$ are exactly the divisors of discrepancy less than $1$ with center $\msq \in Y$ (see \cite[(5.7)]{Reid87}).
We compute
\[
a_X (F_i) = a_Y (F_i) + \frac{1}{2} \ord_{F_i} \psi_i^*E = \frac{i}{5} + \frac{1}{2} \cdot [3i] =
\begin{cases}
\frac{1}{2}, & \text{for $i = 1, 2$}, \\
1, & \text{for $i = 3, 4$}. 
\end{cases}
\]
This shows that $E, F_1$ and $F_2$ are the exceptional divisors discrepancy $1/2$ over $\msp \in X$.
\end{proof}

\begin{Lem} \label{lem:cE6discmin}
Let $\msp \in X$ be the germ at origin of the hypersurface
\[
(f (x, y, z, t) = 0) \subset \mbA^4,
\] 
where $f (x, y, z, t) \in \mbC [x, y, z, t]$.
We assume that $\msp \in X$ is an isolated singularity and
\[
f = x^2 + x z (\lambda t + g_2 (y, z^2)) + g_6 (y, z^2) +  x (t^2 +  h (x, y, z^2)),
\]
where $\lambda \in \mbC$ and $g_3 (y, z^2), g_6 (y, z^2, t), h (x, y, z^2)$ are polynomials satisfying the following properties.
\begin{enumerate}
\item[(a)] The weighted hypersurface
\[
(x^2  + x z (\lambda t + g_2 (y, z^2)) + g_6 (y, z^2) = 0) \subset \mbP (3_x, 2_y, 1_z, 2_t)
\]
is quasismooth outside the point $(0\!:\!0\!:\!0\!:\!1)$.
\item[(b)] For $i = 2, 6$, $g_i (y, z^2)$ a quasi-homogeneous polynomial of degree $i$ with respect to weights $\wt (y, z) = (2, 1)$.
\item[(c)] The polynomial $h (x, y, z^2)$ has weighted order at least $4$ with respect to weights $\wt (x, y, z) = (3, 2, 1)$.
\end{enumerate}
Then $\msp \in X$ is a terminal singularity of type $cE_6$ and the following assertions hold.
\begin{enumerate}
\item If $\lambda \ne 0$, then there are $4$ divisors of discrepancy $1$ over $\msp \in X$.
\item If $\lambda = 0$, then there are $3$ divisors of discrepancy $1$ over $\msp \in X$.
\item There is no divisorial contraction with center $\msp \in X$ of discrepancy greater than $1$.
\end{enumerate}
\end{Lem}

\begin{proof}
Let $\msp \in S$ be a general hyperplane section of $\msp \in X$ cut by the equation $t = \alpha x + \beta y + \gamma z$ for general $\alpha, \beta, \gamma \in \mbC$.
Then the least weight term of $g := f (x, y, z, \alpha x + \beta y + \gamma z)$ with respect to the weights $\wt (x, y, z) = (6, 4, 3)$ is 
\[
x^2 + \lambda \gamma x z^2 + g_6 (y, 0) + \gamma^2 x z^2.
\] 
The assumption (a) in particular implies that $y^3 \in g_6$.
By a suitable coordinate change, we see that $\msp \in S$ is isomorphic to the germ at origin of the hypersurface defined by
\[
x^2 + y^3 + z^4 = 0.
\]
By \cite[Corollary 4.7]{Pae24}, $\msp \in S$ is a Du Val singularity of type $E_6$, and hence $\msp \in X$ is a terminal singularity of type $cE_6$.

Let $\varphi \colon Y \to X$ be the weighted blowup of $\msp \in X$ with weights $\wt (x, y, z, t) = (3, 2, 1, 2)$ and let $E$ be its exceptional divisor.
We have an isomorphism
\[
E \cong (x^2 + x z (\lambda t + g_2 (y, z^2)) + g_6 (y, z^2) = 0) \subset \mbP (3_x, 2_y, 1_z, 2_t).
\]
Let $\msr \in E$ be the point corresponding to $(0\!:\!0\!:\!0\!:\!1)$.
By the assumption (a), the point $\msr$ is the unique singular point of $Y$ along $E$.
We have $a_X (E) = 1$.

We work on the $t$-chart $U_t \subset Y$ of the weighted blowup $\varphi$:
\[
U_t \cong (\tilde{f} (x, y, z, t) = 0)/\mbZ_2 (1_x, 0_y, 1_z, 1_t),
\]
where
\[
\begin{split}
\tilde{f} (x, y, z, t) &:= f (x t^3, y t^2, z t, t^2)/t^6 \\
&= x^2 + x z (\lambda + g_2 (y, z^2)) + g_6 (y, z^2) + x t + x t \tilde{h} (x, y, z, t)
\end{split}
\]
with
\[
\tilde{h} (x, y, z, t) := h (x t^3, y t^2, z^2 t^2)/t^4 \in (x, y, z).
\]
Filtering off terms divisible by $x$ in the equation $\tilde{f} = 0$ and introducing a new variable $s$, we obtain a re-embedding of $U_t$:
 \begin{equation} \label{eq:Utcodim2}
 U_t \cong 
\left(
\begin{tabular}{l}
$x s + g_6 (y, z^2) = 0$, \\
$s - (x + \lambda z +  z g_2 (y, z^2) + t + t \tilde{h}) = 0$
\end{tabular}
\right)/\mbZ_2 (1_x, 0_y, 1_z, 1_t, 1_s).
\end{equation}

If we think of $\msp \in X$ as an analytic germ, then we can eliminate the variable $t$ since $\tilde{h} \in (x, y, z)$.
It follows that the analytic germ $\msp \in X$ is equivalent to the germ at origin of the hyperquotient
\[
(x s + g_6 (y, z^2) = 0)/\mbZ_2 (1_x, 1_s, 0_y, 1_z),
\]
and hence the singularity $\msp \in X$ is a terminal singularity of type $cA/2$.

Let $F$ be a divisor of discrepancy $1$ over $\msp \in X$ other than $E$.
Then the center $C_Y (F)$ of $F$ in $Y$ is a proper closed subvariety of $E$.
We have $a_F (X) = a_F (Y) + \ord_F (E) = 1$, and hence $0 < a_F (Y) < 1$ and $0 < a_F (E) < 1$ since $Y$ has only terminal singularities.
It follows that $C_Y (F)$ is a non-Gorenstein singular point of $Y$ at which $E$ is not Cartier, that is, $C_Y (F) = \msr$.
The discrepancy of a divisor over $\msr \in Y$ is a positive half integer since $2 K_Y$ is Cartier.
It follows that $F$ must be a divisor of discrepancy $1/2$ over $\msr \in Y$.
By Lemma \ref{lem:cA/2divsmin}, there are $3$ distinct divisors of discrepancy $1/2$ over $\msr \in Y$.
It remains to check if each one of them is a divisor of discrepancy $1$ over $\msp \in X$. 

The point $\msr$ corresponds to the origin of $U_t$ and $E|_{U_t}$ is defined by $t$. 
For $i = 1, 3, 5$, let $\psi_i \colon W_i \to Y$ be the weighted blowup of $\msr \in Y$ with weights $\wt (x, y, z, t, s) = (i, 2, 1 ,1 , 5-i)$ and let $F_i$ be its exceptional divisor.
We have isomorphisms
\[
\begin{split}
F_1 &= (x s + g_6 (y, z^2) = t + \lambda z + x = 0)  \subset \mbP (1_x, 2_y, 1_z, 1_t, 5_s), \\
F_3 &= (x s + g_6 (y, z^2) = t + \lambda z = 0) \subset \mbP (3_x, 2_y, 1_z, 1_t, 3_s), \\
F_5 &= (x s + g_6 (y, z^2) = s - (t + \lambda z) = 0) \subset \mbP (5_x, 2_y, 1_z, 1_t, 1_s).
\end{split}
\]
This in particular shows that $F_i$ is a prime divisor over $\msr \in Y$.
We have $a_Y (F_i) = 1/2$ for $i = 1, 3, 5$ and hence $F_1, F_3, F_5$ are the divisors of discrepancy $1/2$ over $\msr \in Y$.

It is easy to see that $\ord_{F_i} (\psi_i^*E) = 1/2$ if $i = 1, 5$, and hence
\[
a_X (F_i) = a_Y (F_i) + \ord_{F_i} (\psi_i^* E) = \frac{1}{2} + \frac{1}{2} = 1
\]
for $i = 1, 5$.
We compute $a_X (F_3)$.
By the equations \eqref{eq:Utcodim2}, the section $t$ can be written as
\[
t (1 + \tilde{h}) = s - (x + \lambda z + z g_2 (y, z^2)),
\]
and $1 + \tilde{h}$ does not vanish at $\msr$ since $\tilde{h} \in (x, y, z)$.
It follows that
\[
\ord_{F_3} (\psi_3^*E) =
\begin{cases}
1/2, & \text{if $\lambda \ne 0$}, \\
3/2, & \text{if $\lambda = 0$}.
\end{cases}
\]
Thus, we have
\[
a_X (F_3) = a_Y (F_3) + \ord_{F_3} (\psi^*_3 E) = 
\begin{cases}
1, & \text{if $\lambda \ne 0$}, \\
2, & \text{if $\lambda = 0$}.
\end{cases}
\]
Therefore, $E, F_1, F_3$ and $F_5$ (resp.\ $E, F_1$ and $F_5$) are the divisors of discrepancy $1$ over $\msp \in X$ if $\lambda \ne 0$ (resp.\ $\lambda = 0$).

It remains to prove (3).
If $\lambda \ne 0$, then the assertions follows from (2) and \cite[Proposition~3.16]{OkSolid}.
Suppose that $\lambda = 0$.
Then the singularity $\hat{\msq} \in \hat{X}$ is equivalent to the germ defined by a normal form
\[
x^2 + y^3 + y \phi (z, t) + h (z, t) = 0,
\]
where $\phi (z, t), h (z, t) \in \mbC \{z, t\}$ are convergent power series of order at least $3, 4$, respectively.
The degree $4$ part of $h (z, t)$ is $t^4$.
The assertion then follows from \cite[Proposition~3.16]{OkSolid}.
\end{proof}

\section{Exclusion for $X_7 \subset \mbP (1, 1, 1, 2, 3)$}
\label{sec:hypdeg7}

Let $X = X_7 \subset \mbP (1, 1, 1, 2, 3)$ be a Fano $3$-fold weighted hypersurface defined by a quasi-homogeneous polynomial $\msF = \msF (x, y, z, t, w)$ of degree $7$, where $x, y, z, t, w$ are homogeneous coordinates of weights $1, 1, 1, 2, 3$, respectively.

\begin{Lem} \label{lem:excl1/3}
Suppose that $X$ is quasismooth at $\msp_w$ and let
\[
\msF = w^2 \ell (x, y, z) + w f_4 + f_7 = 0
\]
be the defining equation of $X$, where $\ell (x, y, z)$ is a nonzero linear form and $f_i (x, y, z, t)$ is a quasi-homogeneous polynomial of degree $i$. 
Let $\varphi \colon Y \to X$ be the Kawamata blowup of $\msp_w \in X$.
Then the following hold.
\begin{enumerate}
\item If $\ell \nmid f_4$, then there is a birational involution $\iota \colon X \ratmap X$ which is an elementary self-link initiated by $\varphi$.
\item If $\ell \mid f_4$, then  $\varphi$ is not a maximal extraction.
\end{enumerate}
\end{Lem}

\begin{proof}
Let $X \ratmap \mbP (1, 1, 1, 2)$ be the projection from the point $\msp_w$.
Then $\varphi$ resolves the indeterminacy of the projection and we obtain a morphism $Y \to \mbP (1, 1, 1, 2)$.
Let
\[
Y \xrightarrow{\psi} Z \to \mbP (1, 1, 1, 2)
\]
be the Stein factorization, where $Z$ is the hypersurface
\[
Z = (s^2 + s f_4 + \ell f_7 = 0) \subset \mbP (1_x, 1_y, 1_z, 2_t, 4_s) =: \overline{\mbP},
\]
$\psi \colon Y \to Z$ is a birational morphism tha is nothing but the anticnaonical morphism, and $Z \to \mbP (1, 1, 1, 2)$ is a double cover.
Let $\Delta$ be the proper transform of the subvariety $(\ell = f_4 = f_7 = 0)_{\mbP} \subset X$ on $Y$.
The morphism $\psi$ contracts $\Delta$ to the set $\psi (\Delta) = (s = \ell = f_4 = f_7 = 0)_{\overline{\mbP}} \subset Z$.
We see that $\ell \nmid f_7$ by the $\mbQ$-factoriality of $X$, so that $\dim \psi (\Delta) \leq 1$.
It is then easy to see that $\Delta$ is a divisor if and only if $\ell \mid f_4$.
The assertions then follow immediately from \cite[Lemma 3.2]{OkII}.
\end{proof}

In the rest of this section, we consider a Gorenstein singular point $\msp \in X$ and we exclude suitable weighted blowups with center $\msp \in X$ as a Sarkisov center.
By a coordinate change, we may assume that $\msp = \msp_x$ and we introduce the following condition.
Let $\msF = \msF (x, y, z, t, w)$ be the defining polynomial of $X$.

\begin{Cond} \label{cond}
\begin{enumerate}
\item The point $\msp_x$ is contained in $X$.
\item We define a weight $\bm{w}_1$ on the variables $x, y, z, t, w$ as follows:
\[
\bm{w}_1 (x, y, z, t, w) = (0, 4, 1, 2, 1).
\]
Then $\bm{w}_1 (\msF) = 6$ and the polynomial $\msF_{\bm{w}_1 = 6}$ is irreducible.
\item We define a weight $\bm{w}'_2$ on the variables $x, y, z, t, w$ as follows:
\[
\bm{w}'_2 (x, y, z, t, w) = (0, 2, 1, 2, 1).
\]
Then $\bm{w}'_2 (\msF) = 4$, $\msF_{\bm{w}'_2 = 4} = \alpha x^5 y^2 + \beta y w^2$ for some $\alpha, \beta \in \mbC \setminus \{0\}$, $\msF_{\bm{w}'_2 = 5} = 0$ and $t^3 \in \msF$.
\end{enumerate}
\end{Cond}

\begin{Lem} \label{lem:condeq}
Suppose that $X \subset \mbP (1, 1, 1, 2, 3)$ satisfies \emph{Condition~\ref{cond}}.
Then the following assertions hold.
\begin{enumerate}
\item $\msF \in (y, z, t)$.
\item We can write 
\[
\msF_{\bm{w}_1 = 6} = \beta w^2 y + \gamma x^2 y z w + x^3 y g_2 (z^2, t) + x g_6 (z^2, t) \in (x, w),
\] 
where $\beta \ne 0, \gamma \in \mbC$ and $g_2 (z^2, t), g_6 (z^2, t)$ are quasi-homogeneous polynomials of degree $2, 6$, respectively, with respect to the weight $\wt (z, t) = (1, 2)$.
Moreover, the polynomial $g_6 (z^2, t)$ is nonzero.
\item We can write
\[
\msF_{\bm{w}'_2 = 6} = y \msH (x, y, z, t, w) + x g_6 (z^2, t) \in (x, w),
\]
where $\msH (x, y, z, t, w)$ and $g_6 (z^2, t)$ are quasi-homogeneous polynomials of degree $4$ and $6$, respectively with respect to the $\bm{w}'_2$-weight.
Moreover, the polynomial $g_6 (z^2, t)$ is nonzero.
\end{enumerate}
\end{Lem}

\begin{proof}
The $\bm{w}_1$-order of any monomial of degree $7$ consisting only of $x$ and $w$ is at most $2$, and hence none of them appear in $\msF$ by (2) of Condition~\ref{cond}.
This proves (1).

The assertion (2) follows by writing down monomials of order $7$ with respect to the weights $(x, y, z, t, w) = (1, 1, 1, 2, 3)$ and of $\bm{w}_1$-order $6$.
The condition $\beta \ne 0$ follows from (3) of Condition~\ref{cond}.
The polynomial $g_6 (z^2, t)$ cannot be zero because otherwise $\msF_{\bm{w}_1 = 6}$ is divisible by $y$ and this is impossible by (2) of Condition~\ref{cond}.

The assertion (3) also follows by writing down monomials of order $7$ with respect to the weights $(x, y, z, t, w) = (1, 1, 1, 2, 3)$ and of $\bm{w}'_2$-order $6$.
The polynomial $g_6 (z^2, t)$ is the same as that in (2), and hence it is nonzero.
\end{proof}

\begin{Lem} \label{lem:hatpsi1}
Suppose that $X$ satisfies \emph{Condition~\ref{cond}} and let $\varphi \colon Y \to X$ be the weighted blowup of $\msp_x \in X$ with weight $\wt (y, z, t, w) = (4, 1, 2, 1)$.
\begin{enumerate}
\item The $\varphi$-exceptional divisor $E$ is a divisor of discrepancy $1$ over $\msp_x \in X$.
\item The anticnaonical divisor $-K_Y$ is in the boundary of the mobile cone $\bMov (Y)$.
\end{enumerate}
In particular, $\varphi$ is not a Sarkisov extraction.
\end{Lem}

\begin{proof}
The polynomial $\msF_{\bm{w}_1 = 6} (1, y, z, t, w)$ is irreducible by (2) of Condition \ref{cond}.
Hence we have an isomorphism
\[
E \cong (\msF_{\bm{w}_1 = 6} (1, y, z, t, w) = 0) \subset \mbP (4_y, 1_z, 2_t, 1_w),
\]
and $E$ is irreducible.
It is then easy to see that $a_X (E) = 1$.
This shows (1).

Let $\Phi \colon \mbT \to \mbP := \mbP (1, 1, 1, 2, 3)$ be the weighted blowup of $\msp_x \in \mbP$ with weights $\wt (y, z, t, w) = (4, 1, 2, 1)$, where $\mbT$ is a rank $2$ toric variety whose description is given in the diagram below.
We identify $Y$ with the proper transform of $X$ in $\mbT$ and $\varphi$ with the restriction $\Phi|_Y$.
We run a $2$-ray game for $\mbT$ and obtain the following diagram.
\[
\xymatrix{
\text{$\mbT := \mbT \begin{pNiceArray}{cc|cccc}[first-row]
u & x & w & z & t & y \\
0 & 1 & 3 & 1 & 2 & 1 \\
-1 & 0 & 1 & 1 & 2 & 4
\end{pNiceArray}$} \ar@{-->}[r]^{\Theta} \ar[d]_{\Phi} &
\text{$\mbT \begin{pNiceArray}{ccc|ccc}[first-row]
u & x & w & z & t & y \\
1 & 1 & 2 & 0 & 0 & -3 \\
1 & 4 & 11 & 3 & 6 & 0
\end{pNiceArray} =: \breve{\mbT}$} \ar[d]^{\breve{\Phi}} \\
\mbP (1_x, 3_w, 1_z, 2_t, 1_y) & \mbP (1_u, 4_x, 11_w, 3_z, 6_t)}
\]
The birational morphisms $\Phi, \breve{\Phi}$, and the birational map $\Theta$ are defined as follows:
\[
\begin{split}
\Phi & \colon (u\!:\!x \, | \, w\!:\!z\!:\!t\!:\!y) \mapsto (x\!:\!w u\!:\!z u\!:\!t u^2\!:\!y u^4), \\
\check{\Phi} & \colon (u\!:\!x\!:\!w \, | \,  z\!:\!t\!:\!y) \mapsto (u y^{1/3}\!:\!x y^{1/3}\!:\!w y^{2/3}\!:\!z\!:\!t), \\
\Theta & \colon (u\!:\!x \, | \, w\!:\!z\!:\!t\!:\!y) \mapsto (u\!:\!x\!:\!w \, | \,  z\!:\!t\!:\!y).
\end{split}
\]
We set $f := F_{\bm{w}_1 = 6}$ which is an irreducible polynomial by (2) of Condition \ref{cond}.
We can write the defining polynomial of $Y$ as
\[
\mcF (u, x, y, z, t, w) := u^{-6} F (x, y u^4, z u, t u^2, w u) = f + u g,
\]
for some polynomial $g = g (u, x, y, z, t, w)$.
We define $\breve{Y} := \Theta_* Y$.
The varieties $Y$ and $\breve{Y}$ are the hypersurfaces in $\mbT$ and $\breve{\mbT}$, respectively, defined by the equation $\mcF = 0$.
By (1) and (2) of Lemma~\ref{lem:condeq}, we have $\msF \in (y, z, t)$ and $\mcF (0, 0, y, z, t, w) = \beta w^2 y$ for some nonzero $\beta \ne 0$.
Then we have
\[
\begin{split}
\Gamma &:= (z = t = y = 0)_{\mbT} \cap Y = (z = t = y = 0)_{\mbT} \subset Y, \\
\breve{\Gamma} &:= (u = x = 0)_{\breve{\mbT}} \cap \breve{Y} = (u = x = y = 0)_{\breve{\mbT}} \subset \breve{Y}.
\end{split}
\]
Both $\Gamma$ and $\breve{\Gamma}$ are irreducible smooth curves.
The restriction $\theta := \Theta|_Y \colon Y \ratmap \breve{Y}$ is a birational map which gives an isomorphism $Y \setminus \Gamma \cong \breve{Y} \setminus \breve{\Gamma}$. 
Let $\breve{X}$ be the image of $\breve{Y}$ by $\breve{\Phi}$, which is a hypersurface in $\mbP (1, 3, 4, 6, 11)$ defined by the equation
\[
\breve{F} := \mcF (u, x, 1, z, t, w) = 0.
\]
The morphism $\breve{\varphi} := \breve{\Phi}|_{\breve{Y}} \colon \breve{Y} \to \breve{X}$ is a birational morphism which contracts the divisor $\breve{E} := (y = 0)_{\breve{\mbT}} \cap \breve{Y}$ to the curve $\breve{C} := (u = x = w = 0)_{\breve{\mbP}}$.
Note that $\breve{C} \subset \breve{X}$ since $\mcF \in (u, x, w)$ by (2) of Lemma~\ref{lem:condeq}, and hence $\breve{F} \in (u, x, w)$.
Thus $Y$ and $\breve{Y}$ are the small $\mbQ$-factorial modifications of $Y$ and we have the decomposition $\bMov (Y) = \Nef (Y) \cup \theta^*\Nef (\breve{Y})$. 

For a variable $\xi \in \{u, x, y, z, t, w\}$, we set $D_{\xi} := (\xi = 0)_{\mbT} \cap Y$ and $\breve{D}_{\xi} := (\xi = 0)_{\breve{\mbT}} \cap \breve{Y}$.
Clearly we have $\theta^*\breve{D}_v = D_v$.
The morphisms $\varphi$ and $\breve{\varphi}$ are defined by some positive multiples of $D_x$ and $\breve{D}_z$, respectively.
It follows that
\[
\begin{split}
\Nef (Y) &= \mbR_{\ge 0} [D_x] + \mbR_{\ge 0} [D_w], \\
\Nef (\breve{Y}) &= \mbR_{\ge 0} [\breve{D}_w] + \mbR_{\ge 0} [\breve{D}_z].
\end{split}
\]
and thus
\[
\bMov (Y) = \mbR_{\ge 0} [D_x] + \mbR_{\ge 0} [D_z].
\]
We have $-K_Y = -\varphi^*K_X - E \sim D_z$.
Thus $-K_Y$ is in the boundary of $\bMov (Y)$.
This shows that $\varphi$ is not a Sarkisov center by \cite[Theorem~3.2]{AZ}.
\end{proof}


We assume that $X$ satisfies Condition \ref{cond}.
We set $\msG := \msF - \msF_{\bm{w}'_2 = 4}$.
Then, by (3) of Condition~\ref{cond}, we can write
\[
\msF = \msF_{\bm{w}'_2 = 4} + \msG = y h + \msG,
\] 
where
\[
h := \alpha x^5 y + \beta w^2
\]
with nonzero $\alpha, \beta \in \mbC$.
We embed $X$ into the weighted projective $5$-space $\mbP (1, 1, 1, 2, 3, 6)$ with homogeneous coordinates $x, y, z, t, w, s$ of weights $1, 1, 2, 3, 6$, respectively, as a codimension $2$ complete intersection subvariety defined by
\begin{equation} \label{eq:reembX}
\begin{split}
\msF_1 (x, y, z, t, w, s) := y s + \msG = 0, \\
\msF_2 (x, y, z, t, w, s) :=  s - h = 0.
\end{split}
\end{equation}

\begin{Lem} \label{lem:hatpsi2}
Under the above setting, let $\varphi \colon Y \to X$ be the weighted blowup of $\msp_x \in X$ with weights $\wt (y, z, t, w, s) = (2, 1, 2, 1, 4)$.
\begin{enumerate}
\item The $\varphi$-exceptional divisor $E$ is a divisor of discrepancy $1$ over $\msp_x \in X$.
\item The anticanonical divisor $-K_Y$ is in the boundary of the mobile cone $\bMov (Y)$.
\end{enumerate}
In particular, $\varphi$ is not a Sarkisov extraction.
\end{Lem}

\begin{proof}
We have an isomorphism
\[
E \cong (y s + \msF_{\bm{w}'_2 = 6} (1, y, z, t, w) = \alpha y + \beta w^2 = 0) \subset \mbP (2_y, 1_z, 2_t, 1_w, 4_s).
\]
We have $\msF_{\bm{w}'_2 = 6} (1, y, z, t, w) = y \msH (1, y, z, t, w) + g_6 (z^2, t)$ and $g_6 (z^2, t) \ne 0$ by (3) of Condition~\ref{lem:condeq}.
Replacing $s$ and eliminating the variable $y$, we have an isomorphism
\[
E \cong (\delta w^2 s + g_6 (z^2, t) = 0) \subset \mbP (1_z, 2_t, 1_w, 4_s),
\]
where $\delta = - \beta/\alpha \in \mbC \setminus \{0\}$.
The polynomial $\delta w^2 s + g_6 (z^2, t)$ is irreducible since $\delta \ne 0$ and $g_6 (z^2, t) \ne 0$.
This implies that $E$ is irreducible and it is then easy to see that $a_X (E) = 1$.
This proves (1).

We prove (2).
Let $\Phi \colon \mbT \to \mbP (1, 1, 1, 2, 3, 6) =: \mbP$ be the weighted blowup of $\msp_x \in \mbP$ with weights $\wt (y, z, t, w, s) = (2, 1, 2, 1, 4)$, where $\mbT$ is a rank $2$ toric variety whose description is given in the diagram below.
We identify $Y$ with the proper transform of $X$ in $\mbT$ and $\varphi$ with $\Phi|_Y$.
We run a $2$-ray game for $\mbT$ and obtain the following diagram. 
\[
\xymatrix{
\text{$\mbT := \mbT \begin{pNiceArray}{cc|ccccc}[first-row]
u & x & w & s & z & t & y \\
0 & 1 & 3 & 6 & 1 & 2 & 1 \\
-1 & 0 & 1 & 4 & 1 & 2 & 2
\end{pNiceArray}$} \ar@{-->}[r]^{\Theta} \ar[d]_{\Phi} &
\text{$\mbT \begin{pNiceArray}{cccc|ccc}[first-row]
u & x & w & s & z & t & y \\
1 & 1 & 2 & 2 & 0 & 0 & -1 \\
1 & 2 & 5 & 8 & 1 & 2 & 0
\end{pNiceArray} =: \breve{\mbT}$} \ar[d]^{\breve{\Phi}} \\
\mbP (1_x, 3_w, 6_s, 1_z, 2_t, 1_y) & \mbP (1_u, 2_x, 5_w, 8_s, 1_z, 2_t)}
\]
The birational morphisms $\Phi, \breve{\Phi}$, and the birational map $\Theta$ are defined as follows:
\[
\begin{split}
\Phi & \colon (u\!:\!x \, | \, w\!:\!s\!:\!z\!:\!t\!:\!y) \mapsto (x\!:\!w u\!:\!s u^4\!:\!z u\!:\!t u^2\!:\!y u^2), \\
\breve{\Phi} & \colon (u\!:\!x\!:\!w\!:\!s \, | \, \!:\!z\!:\!t\!:\!y) \mapsto (u y\!:\!x y^2\!:\!w y^5\!:\!s y^8\!:\!z\!:\!t), \\
\Theta & \colon (u\!:\!x \, | \, w\!:\!s\!:\!z\!:\!t\!:\!y) \mapsto (u\!:\!x\!:\!w\!:\!s \, | \, \!:\!z\!:\!t\!:\!y).
\end{split}
\]
We define
\[
\begin{split}
\mcF_1 (u, x, y, z, t, w, s) &:= u^{-6} \msF_1 (x, y u^2, z u, t u^2, w u, s u^4) \\
 &= y s + u^{-6} \msG (x, y u^2, z u, t u^2, w u), \\
\mcF_2 (u, x, y, z, t, w, s) &:= u^{-2} \msF_2 (x, y u^2, z u, t u^2, w u, s u^4)  \\
&= s u^2 - h (x, y, z, t, w),
\end{split}
\]
and set $\breve{Y} := \Phi_* Y$.
Then $Y$ and $\hat{Y}$ are the complete intersections of codimension $2$ in $\mbT$ and $\breve{\mbT}$, respectively, defined by the equations $\mcF_1 = \mcF_2 = 0$. 
In the above $2$-ray game, the first modification is the map $\Theta' \colon \mbT \ratmap \mbT'$, where
\[
\text{$\mbT' := \mbT \begin{pNiceArray}{ccc|cccc}[first-row]
u & x & w & s & z & t & y \\
0 & 1 & 3 & 6 & 1 & 2 & 1 \\
-1 & 0 & 1 & 4 & 1 & 2 & 2
\end{pNiceArray}$} 
\]
The map $\Theta'$ restricts to an isomorphism on $\mbT \setminus (s = z = t = y = 0)_{\mbT}$.
We have $\msF_1 \in (y, z, t)$ since $\msG \in (y, z, t)$ (1) of Lemma~\ref{lem:condeq}, and we have $w^2 \in h$ by (3) of Condition~\ref{cond}.
Hence
\[
(s = z = t = y = 0)_{\mbT} \cap Y = \emptyset.
\]
This implies that $\Theta'|_Y$ is an isomorphism.
The next modification is the map $\mbT' \ratmap \breve{\mbT}$.
We have $\msF_1 \in (y, z, t)$ as explained above.
Moreover, we have 
\[
\msF_1 (0, x, y, z, t, w, s) = y s + \msF_{\bm{w}'_2 = 6} (x, y, z, t, w)
\]
and $\msF_{\bm{w}'_2 = 6} \in (x, w)$ by (3) of Lemma~\ref{lem:condeq}.
Hence 
\[
\begin{split}
\mcF_1 (u, x, 0, 0, 0, w, s) &= 0, \\
\mcF_2 (u, x, 0, 0, 0, w, s) &= s u^2 - \beta w^2, \\
\mcF_1 (0, 0, y, z, t, 0, s) &= y s, \\
\mcF_2 (0, 0, y, z, t, 0, s) &= 0.
\end{split}
\]
Then we have
\[
\begin{split}
\Gamma &:= (z = t = y = 0)_{\mbT} \cap Y = (z = t = y = s u^2 - \beta w^2 = 0)_{\mbT}, \\
\breve{\Gamma} &= (u = x = w = 0)_{\breve{\mbT}} \cap \breve{Y} = (u = x = w  = y = 0)_{\breve{\mbT}}
\end{split}
\]
Both $\Gamma$ and $\breve{\Gamma}$ are irreducible smooth curves.
The restriction $\theta := \Theta|_Y \colon Y \ratmap \breve{Y}$ is a birational map which gives an isomorphism $Y \setminus \Gamma \cong \breve{Y} \setminus \breve{\Gamma}$.
Let $\breve{X}$ be the image of $\breve{Y}$ by $\breve{\Phi}$, which is a codimension $2$ complete intersections in $\mbP (1_u, 1_z, 2_x, 2_t, 5_w, 8_s)$ defined by the equations
\[
\begin{split}
\breve{F}_1 &= \mcF_1 (u, x, 1, z, t, w, s) = 0 \\
\breve{F}_2 &= \mcF_2 (u, x, 1, z, t, w, s) = 0.
\end{split}
\]
The morphism $\breve{\varphi} := \breve{\Phi}|_{\breve{Y}} \colon \breve{Y} \to \breve{X}$ is a birational morphism which contracts the divisor $\breve{E} := (y = 0)_{\breve{T}} \cap \breve{Y}$ to the curve $\breve{C} = (u = x = w = s = 0)_{\breve{\mbP}} \subset \breve{X}$.

As before, for a variable $\xi \in \{u, x, y, z, t, w\}$, we set $D_{\xi} := (\xi = 0)_{\mbT} \cap Y$ and $\breve{D}_{\xi} := (\xi = 0)_{\breve{\mbT}} \cap \breve{Y}$.
The morphisms $\varphi$ and $\breve{\varphi}$ are defined by some positive multiples of $D_x$ and $\breve{D}_z$, respectively.
It follows that
\[
\begin{split}
\Nef (Y) &= \mbR_{\ge 0} [D_x] + \mbR_{\ge 0} [D_s], \\
\Nef (\breve{Y}) &= \mbR_{\ge 0} [\breve{D}_s] + \mbR_{\ge 0} [\breve{D}_z].
\end{split}
\]
and thus
\[
\bMov (Y) = \mbR_{\ge 0} [D_x] + \mbR_{\ge 0} [D_z].
\]
We have $-K_Y = -\varphi^*K_X - E \sim D_z$.
Thus $-K_Y$ is in the boundary of $\bMov (Y)$ and the proof is complete.
\end{proof}

\section{Proof of Theorem~\ref{mainthm}}
\label{sec:pfmainthm}

\subsection{Analysis of $\hat{X}$}
\label{sec:hatX}

Let $X = X_{12, 14} \subset \mbP (1, 2, 3, 4, 7, 11)$ be a quasismooth weighted complete intersection of type $(12, 14)$.
We recall from \cite{DG23} the construction of an elementary link from $X$ to a Fano $3$-fold hypersurface $\hat{X} = \hat{X}_7 \subset \mbP (1, 1, 1, 2, 3)$ of degree $7$ and give a detailed analysis of $\hat{X}$.

Let $x, y, z, t, v, w$ be the homogeneous coordinates of $\mbP := \mbP (1, 2, 3, 4, 7, 11)$ of weights $1, 2, 3, 4, 7, 11$, respectively.
We set $\msq := \msp_w \in X$.
Note that $\msq \in X$ is a quotient singularity of type $\frac{1}{11} (1, 2, 9)$ and it is the unique singularity of $X$.
Let $\varphi \colon Y \to X$ be the Kawamata blowup at $\msq \in X$.
We can choose $y, t, v$ as local orbifold coordinates of $\msq \in X$ and $\varphi$ is the weighted blowup with weights $\wt (y, t, v) = \frac{1}{11} (1, 2, 9)$.
Let $\Phi \colon \mbT \to \mbP$ be the weighted blowup of $\msq \in \mbP$ with weights $\wt (x, y, z, t, v) = (6, 1, 7, 2, 9)$, where $\mbT$ is the rank $2$ toric variety whose description is given in the diagram below.
We can and do identify $Y$ with the proper transform $\Phi_*^{-1} X \subset \mbT$, and then $\varphi$ coincides with the restriction $\Phi|_Y$.
We run a $2$-ray game from $\mbT$ and observe that it ends with a birational contraction $\hat{\Phi} \colon \hat{\mbT} \to \mbP (1, 1, 1, 2, 3, 6)$ which contracts a divisor $(x = 0)_{\hat{\mbT}}$ to the point $\msp_z$.
These are described in the following diagram.

\[
\xymatrix{
\text{$\mbT := \mbT \begin{pNiceArray}{cc|ccccc}[first-row]
u & w & y & t & v & z & x \\
0 & 11 & 2 & 4 & 7 & 3 & 1 \\
-11 & 0 & 1 & 2 & 9 & 7 & 6
\end{pNiceArray}$} \ar@{-->}[r]^{\Theta} \ar[d]_{\Phi} &
\text{$\mbT \begin{pNiceArray}{ccccc|cc}[first-row]
u & w & y & t & v & z & x \\
3 & 7 & 1 & 2 & 2 & 0 & -1 \\
1 & 6 & 1 & 2 & 3 & 1 & 0
\end{pNiceArray} =: \hat{\mbT}$} \ar[d]^{\hat{\Phi}} \\
\mbP (11_w, 2_y, 4_t, 7_v, 3_z, 1_x) & \mbP (1_u, 6_w, 1_y, 2_t, 3_v, 1_z)}
\]
The birational morphisms $\Phi, \hat{\Phi}$, and the birational map $\Theta$ are defined as follows:
\[
\begin{split}
\Phi & \colon (u\!:\!w \, | \, y\!:\!t\!:\!v\!:\!z\!:\!x) \mapsto (w\!:\!y u^{1/11}\!:\!t u^{2/11}\!:\!v u^{9/11}\!:\!z u^{7/11}\!:\!x u^{6/11}), \\
\hat{\Phi} & \colon (u\!:\!w\!:\!y\!:\!t\!:\!v \, | \,  z\!:\!x) \mapsto (u x\!:\!w x^6\!:\!y x\!:\!t x^2\!:\!v x^3\!:\!z), \\
\Theta & \colon (u\!:\!w \, | \, y\!:\!t\!:\!v\!:\!z\!:\!x) \mapsto (u\!:\!w\!:\!y\!:\!t\!:\!v \, | \,  z\!:\!x).
\end{split}
\]  

We set $\hat{Y} := \Theta_*Y$ and $\hat{X} := \hat{\Phi} (\hat{Y})$.
We also set $\theta := \Theta|_Y \colon Y \ratmap \hat{Y}$ and $\hat{\varphi} := \hat{\Phi}|_{\hat{Y}} \colon \hat{Y} \to \hat{X}$.
We set 
\[
\hat{\msq} := \msp_z \in \mbP (1_u, 1_y, 1_z, 2_t, 3_v, 6_w).
\]
The morphism $\hat{\varphi}$ is a divisorial contraction of discrepancy $1$ over $\hat{\msq} \in \hat{X}$.

We give descriptions of defining equations of these varieties.
Let $\msF_1 (x, y, z, t, v, w)$ and $\msF_2 (x, y, z, t, v, w)$ be the quasi-homogeneous polynomials of degree $12$ and $14$, respectively, with respect to the weight $\wt (x, y, z, t, v, w) = (1, 2, 3, 4, 7, 11)$ which define $X$ in $\mbP$.

\begin{Lem} \label{lem:eqWCI}
By a suitable choice of the homogeneous coordinates $x, y, z, t, v$ and $w$, the defining polynomials $\msF_1, \msF_2$ of $X$ can be written as
\[
\begin{split}
\msF_1 &= -w x + a_{12} (y^2, t) + \lambda y z v + z^4 + z^2 y b_4 (y, t), \\
\msF_2 &= w z + y c_{12} (y^2, t) + v^2 + g_{14} (x, y, z, t),
\end{split}
\] 
where $a_{12} (y^2, t), b_4 (y^2, t), c_{12} (y^2, t)$ and $g_{14} (y, z, t, v)$ are quasi-homogeneous polynomials of the indicated degree with respect to the weights $\wt (x, y, z, t, v, w) = (1, 2, 3, 4, 7, 11)$ satisfying the following properties.
\begin{enumerate}
\item The polynomial $g_{14}$ is contained in the ideal $(x, z)^2$ and has weighted order at least $18/11$ with respect to $\wt (x, y, z, t) = \frac{1}{11} (6, 1, 7, 2)$.
\item The equations $a_{12} (y^2, t) = y c_{12} (y^2, t) = 0$ does not have a nontrivial solution.
\end{enumerate}
\end{Lem}

\begin{proof}
By the quasismoothness of $X$ at the point $\msq = \msp_w$, we have $w x \in \msF_1$ and $w z \in \msF_2$.

We have $v^2 \in \msF_2$ because otherwise $\msp_v \in X$ and $X$ cannot be quasismooth at $\msp_v$.
Replacing $z$ and $v$, we can write $\msF_2 = w z + v^2 + g'_{14} (x, y, z, t)$ for some quasi-homogeneous polynomial $g'_{14} (x, y, z, t)$ of degree $14$.
We can write $g'_{14} (0, y, 0, t) = y c_{12} (y^2, t)$ for some quasi-homogeneous polynomial $c_{12} (y^2, t)$ of degree $12$.
We set $g_{14} := g'_{14} - y c_{12}$, which is contained in the ideal $(x, z)^2$.
It is easy to observe that any monomial in $g_{14}$ has weight at least $18/11$ with respect to the weight $\wt (x, y, z, t) = \frac{1}{11} (6, 1, 7, 2)$.
We have 
\begin{equation}
\label{eq:XdefF2}
\msF_2 = w z + y c_{12} (y^2, z) + v^2 + g_{14} (x, y, z, t),
\end{equation}
with $g_{14} \in (x, z)^2$.

Filtering off terms divisible by $x$ in $\msF_1$ and then replacing $w$, we may write 
\[
\msF_1 = - w x + f_{12} (y, z, t, v)
\] 
for some quasi-homogeneous polynomial $f_{12} (y, z, t, v)$ of degree $12$.
Note that the expression \eqref{eq:XdefF2} of $\msF_2$ may change after this replacement of $w$ in the sense that the terms of the form $v z \phi_4 (x, y, z, t)$ may be added in $\msF_2$ for some quasi-homogeneous polynomial $\phi_4 (x, y, z, t)$ of degree $4$.
By further replacing $v \mapsto v - z \phi_4/2$, we can keep the same expression \eqref{eq:XdefF2} of $\msF_2$.
There is a unique monomial of degree $12$ in variables $y, z, t, v$ of weights $2, 3, 4, 7$, respectively, which is divisible by $v$ and it is $y z v$.
We also set $a_{12} (y^2, t) := f_{12} (y, 0, t, 0)$.
Then we can write 
\[
f_{12} (y, z, t, v) = a_{12} (y^2, t) + \lambda y z v + z f_9 (y, z, t)
\] 
for some homogeneous polynomial $f_9 (y, z, t)$ of degree $9$.
We have $z^4 \in \msF_1$ because otherwise $\msp_z \in X$ and $X$ cannot be quasismooth at $\msp_z$.
Thus we can write $f_9 (y, z, t) = z^3 + z y b_4 (y^2, t)$ for some quasi-homogeneous polynomial $b_4 (y^2, t)$ of degree $4$.
The assertion (2) follows from \cite[Lemma~3.11]{DG23} and the proof is complete.
\end{proof}

The varieties $Y$ and $\hat{Y}$ are complete intersections of codimension $2$ in $\mbT$ and $\hat{\mbT}$, respectively, defined by the equations $\mcF_1 = \mcF_2 = 0$, where
\[
\begin{split}
\mcF_1 (u, x, y, z, t, v, w) &:= u^{-6/11} \msF_1 (x u^{6/11}, y u^{1/11}, z u^{7/11}, t u^{2/11}, v u^{9/11}, w) \\
&=  - w x + a_{12} (y^2, t) + \lambda y z v u + z^4 u^2 + z^2 y b_4 (y^2, t) u, \\
\mcF_2 (u, x, y, z, t, v, w) &:= u^{-7/11} \msF_2 (x u^{6/11}, y u^{1/11}, z u^{7/11}, t u^{2/11}, v u^{9/11}, w) \\
&= w z + y c_{12} (y^2, t) + v^2 u + \tilde{g}_{14},
\end{split}
\]
with
\[
\tilde{g}_{14} := u^{-7/11} g_{14} (x u^{6/11}, y u^{1/11}, z u^{7/11}, t u^{2/11}).
\]
The variety $\hat{X}$ is the WCI of codimension $2$ in $\mbP (1_u, 1_y, 1_z, 2_t, 3_v, 6_w)$ defined by the equations $\hat{\msF}_1 = \hat{\msF}_2 = 0$, where
\[
\begin{split}
\hat{\msF}_1 (u, y, z, t, v, w) &:= \mcF_1 (u, 1, y, z, t, v, w) \\
&=  - w + a_{12} (y^2, t) + \lambda y z v u + z^4 u^2 + z^2 y b_4 (y^2, t) u, \\
\hat{\msF}_2 (u, y, z, t, v, w) &:= \mcF_2 (u, 1, y, z, t, v, w) \\
&= w z + y c_{12} (y^2, t) + v^2 u + \tilde{g}_{14}.
\end{split}
\]
We set $\hat{a}_6 (y^2, t) := a_{12} (y^2, t)$, $\hat{b}_2 (y^2, t) := b_4 (y^2, t)$, $\hat{c}_6 (y^2, t) := c_{12} (y^2, t)$.
By (1) of Lemma~\ref{lem:eqWCI}, we can write 
\[
\tilde{g}_{14} = u^{-7/11} g_{14} (u^{6/11}, y u^{1/11}, z u^{7/11}, t u^{2/11}) = u \hat{g}_6 (u, y, z, t)
\] 
for some polynomial $\hat{g}_6 (u, y, z, t)$.
Note that $\hat{a}_6 (y^2, t), \hat{b}_2 (y^2, t), \hat{c}_6 (y^2, t)$ and $\hat{g}_6$ are quasi-homogeneous polynomials of degree $4, 3, 6$ and $6$ with respect to the weights $\wt (u, y, z, t) = (1, 1, 1, 2)$. 
By eliminating the variable $w$ in terms of the equation $\hat{\msF}_1 = 0$, the variety $\hat{X}$ is isomorphic to the weighted hypersurface of degree $7$ in $\hat{\mbP} := \mbP (1_u, 1_y, 1_z, 2_t, 3_v)$ defined by 
\begin{equation} \label{eq:defeqhatXhyp}
\hat{\msF} := (\hat{a}_6 (y^2, t) + \lambda y z v u + z^4 u^2 + z^2 y u \hat{b}_2 (y^2, t)) z + y \hat{c}_6 (y^2,t) + v^2 u + u \hat{g}_6 = 0.
\end{equation}
Note that the point $\hat{\msq} \in \hat{X}$ is the point $\msp_z \in \hat{\mbP}$.
%

\begin{Lem} \label{lem:singhatX}
The singular points of $\hat{X}$ are $3$ points $\msp_t, \msp_v \in \hat{X}$ and $\hat{\msq}$, which are of type $\frac{1}{2} (1, 1, 1)$, $\frac{1}{3} (1, 1, 2)$ and $cE_6$, respectively.
Moreover, the following hold.
\begin{enumerate}
\item If $\lambda \ne 0$, then there are at most $4$ divisors of discrepancy $1$ over $\hat{\msq} \in \hat{X}$.
\item If $\lambda = 0$, then there are $3$ divisors of discrepancy $1$ over $\hat{\msq} \in \hat{X}$.
\item There is no divisorial contraction with center $\hat{\msq} \in \hat{X}$ of discrepancy greater than $1$.
\end{enumerate}
\end{Lem}

\begin{proof}
By the property (2) of Lemma~\ref{lem:eqWCI}, we have 
\[
(v = z = x = 0)_{\mbT} \cap Y = (x = z = v = a_{12} (y^2, t) = y c_{12} (y^2, t) = 0)_{\mbT} = \emptyset.
\]
This shows that the restriction of the first modification $\Theta' \colon \mbT \ratmap \mbT'$ of the $2$-ray game $\mbT \ratmap \hat{\mbT}$ to $Y$ is an isomorphism, where
\[
\text{$\mbT' := \mbT \begin{pNiceArray}{cccc|ccc}[first-row]
u & w & y & t & v & z & x \\
0 & 11 & 2 & 4 & 7 & 3 & 1 \\
-11 & 0 & 1 & 2 & 9 & 7 & 6
\end{pNiceArray}$}.
\]

We set
\[
\begin{split}
\Gamma &:= (z = x = 0)_{\mbT} \cap Y = (x = z = a_{12} (y^2, t) = y c_{12} (y^2, t) + v^2 u = 0)_{\mbT}, \\
\hat{\Gamma} &:= (u = w = y = t = 0)_{\hat{\mbT}} \cap \hat{Y} = (u = w = y = t = 0)_{\hat{\mbT}}.
\end{split}
\]
Both $\Gamma$ and $\hat{\Gamma}$ are irreducible curves.
The map $\theta \colon Y \ratmap \hat{Y}$ restricts to an isomorphism $Y \setminus \Gamma \cong \hat{Y} \setminus \hat{\Gamma}$.
The variety $Y$ has exactly $2$ singular points $\msr_1$ and $\msr_2$ which are of type $\frac{1}{2} (1, 1, 1)$ and $\frac{1}{9} (1, 2, 7)$, respectively.
We have $\msr_1 \notin \Gamma$ and $\msr_2 \in \Gamma$.
It follows that the image of $\msr_1$ is the unique singular point of $\hat{Y} \setminus \hat{\Gamma}$, and thus $\hat{X} \setminus \hat{\varphi} (\hat{\Gamma})$ has a unique singular point which is of type $\frac{1}{2} (1, 1, 1)$.
It remains to consider singularities of $\hat{X}$ along the curve $\hat{\varphi} (\hat{\Gamma}) = (u = y = t = 0)_{\hat{\mbP}} \subset \hat{X}$.
Here we regard $\hat{X}$ as a hypersurface in $\hat{\mbP}$ defined by the equation \eqref{eq:defeqhatXhyp}.
We have 
\[
\hat{\msF} = v^2 u + h_{14} (u, y, z, t, v),
\]
where $h_{14} (u, y, z, t, v) \in (u, y, t)^2$.
This shows that $\hat{\msq}$ is the unique non-quasismooth point of $\hat{X}$ along $\hat{\varphi} (\hat{\Gamma})$.
Thus $\hat{X}$ has exacatly $3$ singular points $\msp_t, \msp_v$ and $\hat{\msq} = \msp_z$.
The singularities of $\hat{X}$ at $\msp_t$ and $\msp_v$ are of type $\frac{1}{2} (1, 1, 1)$ and $\frac{1}{3} (1, 1, 2)$, respectively.

It remains to determine the singularity type of $\hat{\msq} \in \hat{X}$.
By setting $z = 1$, the point $\hat{\msq}$ corresponds to the origin of the affine hypersurface $\hat{U}$ in $\mbA^4$ defined by the polynomial
\begin{equation} \label{eq:hatf}
\begin{split}
\hat{f} &:= \hat{\msF} (u, y, 1, t, v) \\
&= \hat{a}_6 (y^2, t) + \lambda y v u + u^2 + \hat{b}_2 (y^2, t) u y + y \hat{c}_6 (y^2, t) + v^2 u + u \hat{g}_6 \\
&= u^2 + u y (\lambda v + \hat{b}_2 (y^2, t)) + \hat{c}_6 (y^2, t) + u(v^2 + \hat{g}_6).
\end{split}
\end{equation}
It is easy to see that $\hat{\msq} \in \hat{U}$ satisfies the assumption of Lemma~\ref{lem:cE6discmin}, hence all the assertions follow.
\end{proof}

\begin{Lem} \label{lem:curvedeg1}
There is no curve of degree $1$ on $\hat{X}$ which passes through $\hat{\msq}$ but does not pass through any other singular points.
\end{Lem}

\begin{proof}
Suppose that there exists a curve $C \subset \hat{X}$ of degree $1$ which passes through $\hat{\msq}$ but does not pass through any other singular points.
Let 
\[
\pi \colon \hat{X} \ratmap \mbP (1_u, 1_y, 1_z, 2_t)
\] 
be the projection to the coordinates $u, y, z, t$.
Since $C$ does not pass through the point $\msp_w$, the image $\pi (C)$ is a curve and we have $1 = \deg C = \deg (\pi|_C) \deg (\pi (C))$. 
It follows that $\deg (\pi (C)) = 1/2, 1$.

Suppose that $\deg \pi (C) = 1/2$.
Then $\pi (C) = (u = y = 0)$ since it passes through the point $\pi (\hat{\msq}) = (0\!:\!0\!:\!1\!:\!0)$.
Let $\mu \in \mbC$ be the coefficient of $t^3$ in $\hat{a}_6 (y^2, t) = a_{12} (y^2, t)$, which is nonzero by (2) of Lemma~\ref{lem:eqWCI}.
Then the curve $C$ is contained in the set
\[
(u = y = 0)_{\hat{\mbP}} \cap \hat{X} = (u = y = \mu t^3 z = 0),
\]
and thus $C$ is either $(u = y = t = 0)$ or $(u = y = z = 0)$.
This is absurd.

Suppose that $\deg \pi (C) = 1$.
Then $\deg (\pi|_X) = 1$.
We can write
\[
\pi (C) = (t - z \ell_1 (u, y) - q_1 (u, y) = \ell_2 (u, y) = 0) \subset \mbP (1, 1, 1, 2),
\]
where $\ell_i (u, y), q_1 (u, y)$ are linear and quadratic forms, respectively.
Then $C$ can be written as
\[
C = (\ell_1 = t - z \ell_2 (u, y) - q_1 (u, y) = v - z^2 \ell_3 - z q_2 (u, y) - c (u, y) = 0) \subset \hat{\mbP},
\]
where $\ell_3 (u, y), q_2 (u, y)$ and $c (u, y)$ are linear, quadratic and cubic forms, respectively.
We see that the polynomial
\begin{equation} \label{eq:inclusionhatXC}
\hat{\msF} (u, y, z, z \ell_2 + q_1 , z^2 \ell_3 + z q_2 + c) \in \mbC [u, y, z]
\end{equation}
is divisible by $\ell_1$ since $C \subset \hat{X}$.
The monomial $z^5 u$ is the only monomial in the polynomial \eqref{eq:inclusionhatXC} which is divisible by $z^5$ and thus $\ell_1 = u$.
We then have
\[
\phi (y, z) := \hat{\msF} (0, y, z, z \bar{\ell}_2 + \bar{q}_1 , z^2 \bar{\ell}_3 + z \bar{q}_2 + \bar{c}) = 0,
\]
where $\bar{\ell}_i = \ell_i (0, y), \bar{q}_i = q_i (0, y)$ and $\bar{c} = c (0, y)$.
Since 
\[
\hat{\msF} (0, y, z, t, v) = \hat{a}_6 z + y \hat{c}_6,
\] 
we have $\phi = z^4 \bar{\ell}^3_2 + \cdots$, where the omitted terms is the sum of monomials which are not divisible by $z^4$.
Hence $\bar{\ell}_2 = 0$.
We write $\bar{q}_1 = \alpha y^2$ for some $\alpha \in \mbC$.
Then $\phi (y, z) = 0$ implies that $\hat{a}_6$ and $y \hat{c}_6$ are both divisible by $t - \alpha y^2$.
This is impossible by (2) of Lemma~\ref{lem:eqWCI}.
Thus the proof is complete.
\end{proof}

\begin{Lem} \label{lem:hatXcond}
The weighted hypersurface $\hat{X} \subset \hat{\mbP} = \mbP (1_u, 1_y, 1_z, 2_t, 3_v)$ satisfies \emph{Condition~\ref{cond}}.
\end{Lem}

\begin{proof}
This is straightforward after replacing $(u, y, z, t, v) \mapsto (z, x, z, t, w)$.
\end{proof}

\subsection{Construction of a link in the case $\lambda \ne 0$}
\label{sec:anothtelink}

We keep the settings as in \S~\ref{sec:hatX}.
We show that if $\lambda \ne 0$, then there is an elementary link $\hat{X} \ratmap X$ which is different from $\sigma^{-1}$.

Let $\hat{\chi} \in \Aut (\hat{X})$ be the automorphism of $\hat{X}$ defined by
\[
(u\!:\!y\!:\!t\!:\!v\!:\!z) \overset{\chi}{\mapsto} (u\!:\!y\!:\!t\!:\!-v - \lambda y z^2\!:\!z),
\]
and we set
\[
\hat{\varphi}' := \hat{\chi} \circ \hat{\varphi} \colon \hat{Y} \to \hat{X}.
\]
Suppose that $\lambda \ne 0$.
In this case, to distinguish $\hat{\varphi}$ and $\hat{\varphi}'$, we use the symbols $\hat{Y}' := \hat{Y}$ and $\hat{E}' := \hat{E}$ for $\hat{\varphi}'$, that is, we denote $\hat{\varphi}' \colon \hat{Y}' \to \hat{X}$ and the $\hat{\varphi}'$-exceptional divisor is $\hat{E}'$.
The divisors $\hat{E}$ and $\hat{E}'$ define distinct valuation since $\nu_{\hat{E}} (v) = 2$ and $\nu_{\hat{E}'} (v) = \nu_{\hat{E}} (-v - \lambda y) = 1$.
We define
\[
\sigma' := \hat{\chi}^{-1} \circ \sigma \colon X \ratmap \hat{X},
\]
which is clearly an elementary link (initiated by $\varphi$).
Note that $\sigma'$ and $\sigma$ differ only up to the composite of the automorphism $\hat{\chi}$ of $\hat{X}$. 

\begin{Lem}
\label{lem:anotherlink}
Suppose that $\lambda \ne 0$.
Then, $\hat{\varphi}'$ initiates the elementary link ${\sigma'}^{-1} = \sigma^{-1} \circ \hat{\chi} \colon \hat{X} \ratmap X$ which is different from $\sigma^{-1}$.
Moreover, the composite
\[
X \overset{\sigma}{\ratmap} \hat{X} \overset{{\sigma'}^{-1}}{\ratmap} X
\]
is a birational involution of $X$ which is defined by
\[
(x\!:\!y\!:\!z\!:\!t\!:\!v\!:\!w) \mapsto
 \left(x\!:\!y\!:\!z\!:\!t\!:\!-v - \frac{\lambda y z^2}{x}\!:\!w \right).
 \]
\end{Lem}

\begin{proof}
The valuations of $\hat{E}$ and $\hat{E}'$ are different, hence the associated elementary links $\sigma^{-1}$ and ${\sigma'}^{-1}$ are different.
The rest is straightforward. 
\end{proof}

\subsection{Proof of Theorem~\ref{mainthm}}

We keep the settings as in \S~\ref{sec:hatX}.

\begin{Prop} \label{prop:classiflink}
The following assertions hold.
\begin{enumerate}
\item If $\lambda = 0$, then $\sigma^{-1} \colon \hat{X} \ratmap X$ is the unique elementary link from $\hat{X}$.
\item If $\lambda \neq 0$, then $\sigma^{-1} \colon \hat{X} \ratmap X$ and ${\sigma'}^{-1} \colon \hat{X} \ratmap X$ are the elementary links from $\hat{X}$.
\end{enumerate}
\end{Prop}

\begin{proof}
By \cite[Lemmas~4.5 and 4.9]{OkSolid}, smooth points on $\hat{X}$ and the $\frac{1}{2} (1, 1, 1)$ point $\msp_t \in \hat{X}$ are not maximal centers. 
The defining equation \eqref{eq:defeqhatXhyp} is of the form
\[
v^2 u + \lambda y z u v + \cdots = 0,
\]
where the omitted term is a polynomial which do not involve the variable $v$.
By Lemma~\ref{lem:excl1/3}, the $\frac{1}{3} (1, 1, 2)$ point $\msp_v \in \hat{X}$ is not a maximal center.

We show that no curve on $\hat{X}$ is a maximal center.
Suppose that there is an irreducible curve $C$ on $\hat{X}$ which is a maximal center.
If $C$ passes through a terminal quotient singular point, then there is no divisorial contraction with center $C$ by \cite{Kawamata}, and hence $C$ cannot be a maximal center.
This is a contradiction, and $C$ does not pass through a terminal quotient singular point.
This in particular implies $\deg C := (\hat{A} \cdot C) \in \mbZ$, where $\hat{A} = - K_{\hat{X}}$ is the Weil divisor class such that $\mcO_{\hat{X}} (\hat{A}) \cong \mcO_{\hat{X}} (1)$.  
By \cite[Lemma~2.9]{OkII}, we have $\deg C < 7/6$, and hence $\deg C = 1$.
By Lemma~\ref{lem:curvedeg1}, the curve $C$ does not pass through $\hat{\msq}$, that is, $C$ is contained in the smooth locus of $\hat{X}$.
Then, by completely the same argument as in Step 2 of the proof of \cite[Theorem~5.1.1]{CPR}, the curve $C$ is not a maximal center.
This is a contradiction and we conclude that no curve on $\hat{X}$ is a maximal center.

It remains to check whether a divisorial contraction with center $\hat{\msq} \in \hat{X}$ is a Sarkisov center or not.
By Lemma~\ref{lem:hatXcond}, there are two divisors of discrepancy $1$ over $\hat{\msq} \in \hat{X}$ given in Lemmas~\ref{lem:hatpsi1} and \ref{lem:hatpsi2}.
Let $\hat{\psi}_1 \colon \hat{W}_1 \to \hat{X}$ and $\hat{\psi}_2 \colon \hat{W}_2 \to \hat{X}$ be the weighted blowups given in Lemma~\ref{lem:hatpsi1} and \ref{lem:hatpsi2}, respectively, and let $\hat{F}_i$ be the $\hat{\psi}_i$-exceptional divisor for $i = 1, 2$.
In the case $\lambda \ne 0$, let $\hat{\varphi}' \colon \hat{Y}' \to \hat{X}$ be the divisorial contraction given in \S~\ref{sec:anothtelink}.
It is easy to observe that $\hat{E}, \hat{F}_1, \hat{F}_2$ (and also $\hat{E}'$ in the case $\lambda \ne 0$) define mutually distinct valuations of the function field of $\hat{X}$.
Thus, by Lemma~\ref{lem:singhatX}, these exhaust the divisors of discrepancy $1$ over $\hat{\msq} \in \hat{X}$.
By (3) of Lemma~\ref{lem:singhatX} and \cite[Lemma~3.4]{Kawakita01}, there exist no divisorial contraction with center $\hat{\msq} \in \hat{X}$ other than $\hat{\varphi}, \hat{\psi}_1, \hat{\psi}_2$ and $\hat{\varphi}'$, and $\hat{\psi}_i$ is a divisorial contraction if and only if $\hat{W}_i$ has only terminal singularities.
As it is already explained, the divisorial contractions $\hat{\varphi}$ and $\hat{\varphi}'$ initiate the elementary links $\sigma^{-1}$ and ${\sigma'}^{-1}$, respectively.
If $\hat{\psi}_i$ is a divisorial contraction, then it is not a Sarkisov extraction by Lemmas~\ref{lem:hatpsi1} and \ref{lem:hatpsi2}.
\end{proof}

\begin{proof}[Proof of Theorem~\ref{mainthm}]
By \cite[Corollaries~7.2 and 7.11]{DG23}, no smooth point and no curve on $X$ is a maximal center.
The point $\msq \in X$ is the unique singular point of $X$.
The Kawamata blowup $\varphi \colon Y \to X$ is the unique divisorial contraction with center $\msq \in X$ and it initiates the link $\sigma$.
It follows that $\sigma$ is the unique elementary link from $X$.
Combining this with Proposition~\ref{prop:classiflink}, Theorem~\ref{mainthm} follows immediately.
\end{proof}


\end{document}